\documentclass[11pt]{article}
\usepackage{mathptmx}

\usepackage{amsmath}
\usepackage{amssymb}
\usepackage{amsthm}
\usepackage[all]{xy}
\usepackage{amscd}
\usepackage{wasysym}

\usepackage{url}

\usepackage{euscript}
\usepackage{array}
\usepackage{float}
\usepackage[labelsep=period]{caption}[2006/03/21]

\usepackage{mathrsfs}
\newcommand{\mscr}[1]{\mathscr{#1}}

\usepackage{xspace}
\newcommand{\dvbt}{double vector bundle\xspace}
\newcommand{\dvbts}{double vector bundles\xspace}

\newtheorem{thm}{Theorem}[section]
\newtheorem{lem}[thm]{Lemma}
\newtheorem{prop}[thm]{Proposition}

\theoremstyle{definition}
\newtheorem{defin}[thm]{Definition}
\newtheorem{ej}[thm]{Example}
\theoremstyle{remark}
\newtheorem{rem}[thm]{Remark}

\newcommand{\pf}{\noindent{\textsc{Proof.}\ }}
\newcommand{\pfend} {\hspace*{\fill}\lower3pt\hbox{$\Box$}\medskip} 

\newcommand{\co}{\colon\thinspace}

\newcommand{\Ga}{\Gamma}
\newcommand{\lhangle}{\,\rule[-0.1ex]{0.4ex}{0.7\baselineskip}\,\,}
\newcommand{\rhangle}{\,\,\rule[-0.1ex]{0.4ex}{0.7\baselineskip}\,}
\newcommand{\bk}{\,\vert\, }
\renewcommand{\Bar}[1]{\overline{#1}}
\renewcommand{\Tilde}[1]{\widetilde{#1}}
\renewcommand{\phi}{\varphi}

\newcommand{\pair}[2]{\langle #1\,\vert\,#2\rangle}

\newcommand{\rr}{\mathbb{R}}

\newcommand{\zz}{\mathbb{Z}}

\newcommand{\sol}{\bullet}
\newcommand{\Lra}{\longrightarrow}
\newcommand{\id}{\operatorname{id}}
\newcommand{\End}{\operatorname{End}}
\DeclareMathOperator{\Dec}{Dec}
\DeclareMathOperator{\Stat}{Stat}
\renewcommand{\leq}{\leqslant}

\newcommand{\BE}{E_\sol} 

\newcommand{\bz}{\mathop{*}} 
\newcommand{\ABC}{{A \bz B \bz C}}

\newcommand{\Wii}{G_2} 
\newcommand{\decom}[1]{\overline{#1}}
\newcommand{\duer}{\mbox{${}^\times\kern-.4em\rule{0.3pt}{8pt}\,\,$}} 
\newcommand{\du}{^{\star}}
\newcommand{\ts}[1]{#1}
\newcommand{\C}{\mscr{C}}

\newcommand{\un}[2]{\underset{#2}{#1}}

\newcommand{\tilo}{\widetilde 0}           
\newcommand{\cinf}[1]{C^{\infty}(#1)}

\newdir{ >}{{}*!/-5pt/@{>}}

\usepackage{subfig}

\newcommand{\add}[1]{\mathbin{\lower 7pt%
    \hbox{${\stackrel{\textstyle +}{\scriptscriptstyle #1}}$}}}

\newcommand{\dg}{\mscr{D}\!\mscr{F}} 
\newcommand{\W}{\mscr{W}} 
\newcommand{\dvb}[8]{\xymatrix{#1 \ar[r]^{#6} \ar[d]_{#5} & #3 \ar[d]^{#8} \\ #2 \ar[r]^{#7} & #4}}
\newcommand{\dvbs}[4]{\dvb{#1}{#2}{#3}{#4}{}{}{}{}}
\newcommand{\vb}[3]{\xymatrix{#1 \ar[r]^{#3} & #2}}

\newcommand{\emb}[3]{\xymatrix{#1 \ar@{^{(}->}[r]^{#3} & #2}}

\newcommand{\dvbc}[5]{\xymatrix@=8pt{ #1 \ar[rr] \ar[dd] && #3 \ar[dd] \\ & #4 \ar@{_{(}->}[ul] \\  #2 \ar[rr] && #5}} 

\newcommand{\ses}[3]{\SelectTips{cm}{12}%
  \xymatrix@=8mm{#1\ar@{ >->}[r] & #2 \ar@{->>}[r] & #3}}
\newcommand{\smallcube}[8]{ \xymatrix@=5mm{ #1 \ar[rr] \ar[dd] \ar[rd] && #2 \ar[rd] \ar'[d][dd] \\ & #3 \ar[rr] \ar[dd] && #4 \ar[dd] \\ #5 \ar[rd] \ar'[r][rr] && #6 \ar[rd] \\ & #7 \ar[rr] && #8}}

\begin{document}
\title{\textbf{Duality functors for triple vector bundles}\footnote{Mathematics Subject 
Classification (MSC2000): 53D17 (primary), 18D05, 18D35, 20E99, 55R99, 70H50 (secondary).}
} 

\author{
\begin{tabular}{lcl}
Alfonso Gracia-Saz & and \quad & K. C. H. Mackenzie\\
Department of Mathematics                   & &        School of Mathematics and Statistics\\
University of Toronto                   & &        University of Sheffield\\
Toronto ON, M5S 2E4                   & &        Sheffield, S3 7RH\\
Canada                  & &        United Kingdom\\
\url{alfonso@math.toronto.edu} & &   \url{K.Mackenzie@sheffield.ac.uk}
\end{tabular}
}

\date{July 5, 2009}

\maketitle

\begin{abstract}
We calculate the group of dualization operations for triple vector bundles, showing
that it has order 96 and not 72 as given in Mackenzie's original treatment. The group
is a nonsplit extension of $S_4$ by the Klein group. 
Dualization operations are interpreted as functors on appropriate categories and
are said to be equal if they are naturally isomorphic. The method set out here will
be applied in a subsequent paper to the case of $n$-fold vector bundles. 
\end{abstract}


\section*{Introduction.}

The duality of finite-rank vector bundles $E$ is involutive, $(E^*)^*\cong E$, and so
may be said to have group $\zz_2$. In \cite{Mackenzie:1999} and \cite{Mackenzie:2005dts}, 
Mackenzie initiated the study of duality for multiple vector bundles, showing that the
group of dualization operations of a double vector bundle is the symmetric group of
order six, and stating that the corresponding group for triple vector bundles has
order 72. 

Shortly before \cite{Mackenzie:2005dts} went to press, Gracia-Saz convinced Mackenzie 
that the order of this group is 96. In the present paper we prove that this is indeed
so, by a method which we will subsequently extend to  $n$-fold vector 
bundles \cite{Gracia-SazM:nfold}. The approach in \cite{Mackenzie:2005dts}
considered only the triple vector bundles themselves; in this paper, we show that the 
correct question to pose is when two dualization operations are equal. The crux 
of the method is to regard dualization operations as functors on appropriate categories, 
and to regard two dualization operations as equal if they are naturally isomorphic as 
functors.   

One of the motivating examples at the inception of category theory \cite{MacLane:pres}
was the distinction between the isomorphisms which can be defined between a (finite-dimensional) 
vector space $V$ and its dual by use of a basis, and the natural isomorphism which exists
between $V$ and $(V^*)^*.$ Our method extends this familiar idea to multiple vector 
bundles. For ordinary vector bundles, local trivializations are transformed into each
other by maps into a general linear group, and dualizing the bundle corresponds to taking
the transpose in the group. Multiple vector bundles can always be globally decomposed --- that 
is, they are always isomorphic to a combination of pullbacks of ordinary vector bundles ---  and 
there is a corresponding group of what we call statomorphisms (see Definition \ref{df:stato3}). 
In order to distinguish between two dualization operations for triple vector bundles 
it is not sufficient to consider the action on the constituent bundles, as was done in 
\cite{Mackenzie:2005dts} --- it is necessary to study their effect on the statomorphism group. 

Double vector bundles have been used for many years in some accounts of connection theory
\cite{Besse} and in some approaches to theoretical mechanics \cite{Tulczyjew}. Their general 
theory was initiated by Pradines \cite{Pradines:DVB}. In the 1990s double and triple vector
bundles began to arise in Poisson geometry, as a result of the relationships between Poisson
structures and Lie algebroids. At the simplest level, the dual of a Lie algebroid $A$ has a 
Poisson structure, the linearity properties of which can be expressed as the condition 
that the Poisson anchor $\pi^\#$ is a morphism of double vector bundles $T^*A^*\to TA^*$. 
Secondly, given a Lie bialgebroid $(A,A^*)$, the cotangent double $T^*A^*\cong T^*A$ can play 
a role corresponding to that of the classical Drinfel'd double of a Lie bialgebra. More 
generally, the duality of double vector bundles, and the triple vector bundles associated with 
them, are crucial to the compatibility conditions of \cite{Mackenzie:eddd} between the Lie 
algebroid structures in a double Lie algebroid. The results of the present paper will be applied 
elsewhere to the triple Lie algebroids which arise from double Lie algebroids, and in particular 
to understanding the duals of thrice-iterated tangent bundles $T^3M$. 
No familiarity with Lie algebroids or Poisson structures is needed in this paper. 

The duality group for double vector bundles $(D;A,B;M)$ is the symmetric group $S_3$ and duality 
operations in this case can be identified by their action on the three `building bundles': the 
side bundles $A$ and $B$ and the core dual $C^*$. Indeed the duality group for double vector bundles 
can be interpreted as those rotations of the cotangent triple $T^*D$ which preserve $T^*D$ 
and $M$ \cite{Mackenzie:2005dts}. In the triple case the corresponding result does not hold: there 
are three nontrivial duality operations which preserve the building bundles but are not equivalent 
to the identity. The significance of these operations deserves to be investigated further. 

In \S\ref{sect:intro} we briefly review the basic notions for double vector bundles; some
readers may prefer to start with \S\ref{sect:df}. To introduce the methods of this paper and 
of \cite{Gracia-SazM:nfold}, we begin by illustrating them on the known case of double vector 
bundles. The duality theory in this case was set out in  \cite[\S1--\S3]{Mackenzie:2005dts} 
but we reformulate it in \S\ref{sect:df} to clarify the problems which arise in the case of 
triple vector bundles and to introduce the functorial point of view which we use in this 
paper and in \cite{Gracia-SazM:nfold}. In Definition \ref{defin:dg2} we define the group $\dg_2$, 
the group of duality functors modulo natural equivalence, which we usually refer to simply
as the \emph{duality group for double vector bundles}.

In \S\ref{sect:tvb} we set up a notation for triple vector bundles that will extend readily 
to the $n$-fold case. The main work of the paper is in \S\ref{sect:dtvb}. This is the calculation
of the duality group $\dg_3$ for triple vector bundles, in terms of its action on the group of 
statomorphisms (Definition \ref{df:stato3}). The final \S\ref{sect:group} provides some further 
information on the structure of $\dg_3$. Throughout the paper we consider smooth vector bundles 
of finite rank over the reals. 

Multiple vector bundles in the setting of supergeometry are being developed by Voronov 
\cite{Voronov:mtqm} and Grabowski \cite{GrabowskiR:mvbs}. 

Our work on this paper began after Mackenzie had spoken on \cite{Mackenzie:2005dts} at
Poisson 2004 in Luxembourg. Gracia--Saz's research was partially supported by fellowships 
from the \emph{Secretar\'ia de Estado de Universidades e Investigaci\'on del Ministerio 
Espa\~{n}ol de Educaci\'on y Ciencia} and from the Japanese Society for the Promotion of 
Science. Gracia--Saz would also like to thank the Pure Mathematics department of the 
University of Sheffield for its hospitality during various visits, and both authors thank
the London Mathematical Society for funding one of these visits. The authors also very much
appreciate the comments of the referees.


\section{Review of double vector bundles} 
\label{sect:intro}

A \emph{double vector bundle} $(D;A,B;M)$, as shown on the left of Figure~\ref{fig:dvbdef}, 
consists firstly of a manifold $D$ together with two vector bundle structures, on bases $A$ and $B$, 
each of which is itself a vector bundle on base $M$, such that for each structure on $D$, 
the structure maps (projection, addition, scalar multiplication) are vector bundle 
morphisms with respect to the other structure. A \emph{morphism of double vector bundles} 
from $(D;A,B;M)$ to $(D';A',B';M')$ is a system of maps 
$\phi\co D\to D'$, $\phi_A\co A\to A'$, $\phi_B\co B\to B'$, $\phi_M\co M\to M'$ 
such that each of $(\phi,\phi_A)$, $(\phi,\phi_B)$, $(\phi_A,\phi_M)$ and $(\phi_B, \phi_M)$
is a morphism of vector bundles. 

Two examples to keep in mind, also shown in Figure~\ref{fig:dvbdef}, are the tangent prolongation 
of an ordinary vector bundle $A$, and the double vector bundle $A\bz B$ which is formed from two 
vector bundles $A$ and $B$ on $M$ by giving the pullback manifold $A\times_M B$ the pullback vector 
bundle structures $q_A^!B$ and $q_B^!A$. (We denote the pullback of $B$ over $q_A$ by $q_A^!B$ 
instead of $q_A^\star B$, so as not to confuse the symbol $\star$ with the many duals that appear 
in this paper.) We use the notation $A \bz B$ to distinguish this double vector bundle from the 
Whitney sum $A \oplus B$;  as manifolds they are the same, but we regard $A\oplus B$ as a vector 
bundle over the base $M$, and $A\bz B$ as a vector bundle over the bases $A$ and $B$. The tangent 
prolongation vector bundle $TA\to TM$ is formed by applying the tangent functor to the structure 
maps of $A\to M$. 

\begin{figure}[h]
\begin{picture}(340,80)(-20,-60) 
\put(30,0){\dvb{D}{A}{B}{M}{q^D_A}{q^D_B}{q_A}{q_B}}
\put(130,0){\dvb{TA}{A}{TM}{M}{}{}{q_A}{}}
\put(230,0){\dvb{A \bz B}{A}{B}{M}{}{}{q_A}{q_B}}
\end{picture}
\caption{\label{fig:dvbdef}}
\end{figure}

We also need a further condition, which was not made explicit in \cite{Mackenzie:2005dts}
or \cite{Mackenzie:GT}. 

\begin{defin}
A double vector bundle $(D;A,B;M)$ satisfies the \emph{splitting condition} if there is a
morphism of double vector bundles $\Sigma\co A\bz B\to D$ which preserves $A$ and $B$; such
a map is a \emph{splitting of} $D$. 
\end{defin}

Equivalently, a splitting is a map $\Sigma\co A\times_M B\to D$ which is right-inverse
to the combination of the two projections $D\to A\times_M B$ and, when regarded as 
$q_A^!B\to D$, is linear over $A$ and, when regarded as $q_B^!A\to D$, is linear over $B$. 
We note below that the splitting condition is equivalent to the existence of a 
decomposition, and as such it is the counterpart of requesting local triviality in the 
definition of vector bundle.  We could request only a local splitting condition, but a standard 
\v{Cech} cohomology argument shows that this implies the global splitting condition.
It can be shown that for double vector bundles the splitting condition is implied by the rest 
of the definition \cite{GrabowskiR:mvbs}. In this paper, we include it as part of the definition.

It is standard that connections in a vector bundle $A\to M$ correspond to maps $A\bz TM\to TA$ 
which are linear in each variable and preserve $A$ and $TM$ (see Example \ref{ej:conns}). It is 
easy to check (see below Definition \ref{defin:decom}) that the duals of a double vector 
bundle which satisfies the splitting condition, also satisfy the condition. 


We briefly recall the basic constructions with double vector bundles as used in 
\cite[Chap.~9]{Mackenzie:GT}. In terms of $(D;A,B;M)$ as in Figure~\ref{fig:dvbdef}, 
we refer to $A$ and $B$ as the \emph{side bundles} of $D$, and to $M$ as the \emph{double 
base}. In the two side bundles the addition, scalar multiplication and subtraction are 
denoted by the usual symbols $+$, juxtaposition, and $-$. We distinguish the two 
zero-sections, writing $0^A \co M\to A,\ m\mapsto 0^A_m$, and $0^B \co M\to B,\ m\mapsto
0^B_m$. We may denote an element $d\in D$ by $(d;a,b;m)$ to indicate that 
$a = q^D_A(d),\ b = q^D_B(d),\ m = q_B(q^D_B(d)) = q_A(q^D_A(d)).$

In the \emph{vertical bundle structure} on $D$ with base $A$ the vector bundle operations 
are denoted $+_A, {.}_A, {-}_A$, with $\tilo^A\co A \to D,\ a\mapsto\tilo^A_a$, for 
the zero-section. In the \emph{horizontal bundle structure} on $D$ with base $B$ we 
likewise write $+_B$ and so on. For $m\in M$ the \emph{double zero}
$\rule[-8pt]{0pt}{22pt}\tilo^A_{0^A_m} = \tilo^B_{0^B_m}$ is denoted 
$\odot_m$ or $0^2_m$. 

The map $D\to A*B$ which combines the two bundle projections is a morphism of double
vector bundles and the set of elements of $D$ which map to the double zeros of $A*B$ 
is denoted $C$ and called the \emph{core of} $D$. Thus 
$$
C=\{c \in D ~|~ \exists m\in M \mbox{\ such that}\ q^D_B(c )= 0^B_m,\
q^D_A(c)=0^A_m\}.
$$
The core is an embedded submanifold of $D$ and has a well-defined vector bundle 
structure on base $M$. The projection $q_C$ is the restriction of 
$q_B\circ q^D_B = q_A\circ q^D_A$ and the addition and scalar multiplication 
are the restrictions of either of the operations on $D$. This is the \emph{core 
vector bundle of} $D$. 

As an ordinary vector bundle, $D\to B$ has a dual which we denote $D^X$
or $D\duer B$. There is a second vector bundle structure on $D^X$, with base $C^*$,
and projection $\zeta$ given by
\begin{equation}                             \label{eq:unfprojH}
\langle \zeta(\Psi), c\rangle =
   \langle\Psi, \tilo^B_b +_B c\rangle
\end{equation}
where $c\in C_m,\ \Psi\co (q^D_B)^{-1}(b)\to\rr$ and $b\in B_m.$ The
addition $+_{C^*}$ and scalar multiplication in 
$D^X\to C^*$ are defined by
\begin{gather}                              
\langle\Psi +_{C^*}\Psi', d +_A d'\rangle =
   \langle\Psi,d\rangle + \langle\Psi',d'\rangle\\
\langle t ._{C^*}\Psi, t ._A d\rangle = t\langle\Psi,d\rangle,
\end{gather}
for suitable elements. The zero above $\kappa\in C^*_m$ is denoted $\tilo^{\du B}_\kappa$
and is defined by
\begin{equation}
\langle\,\tilo^{\du B}_\kappa, \tilo^A_a +_B c\rangle =
\langle\kappa, c\rangle
\end{equation}
where $a\in A_m,\ c\in C_m.$ The core of $D^X$ is $A^*$ with $\phi\in A_m^*$ 
defining $\Bar{\phi}\in D^X$ by
$$
\langle\Bar\phi, \tilo^A_a +_B c\rangle =
\langle\phi,a\rangle.
$$
With these structures $(D^X;C^*,B;M)$ is a double vector bundle with core
$A^*$. Likewise the dual of $D\to A$, denoted $D^Y$ or $D\duer A$, has a vector 
bundle structure over $C^*$, making $(D^Y;A,C^*;M)$ a double vector bundle with 
core $B^*$. These are displayed in Figure~\ref{fig:dvb}. 

Observe that the notation $(D;A,B;M)$, and the diagram in the Figure, indicate
that a choice has been made between the two structures on $D$. We may say that such 
a double vector bundle has been \emph{placed}, and that the \emph{flip} of $D$, 
namely $(D^f;B,A;M)$ in which the two structures on $D$ have been interchanged, 
has the opposite \emph{placement}. 

\begin{figure}[h]
\begin{picture}(340,80)(-20,-60) 
\put(30,0){\dvbc{D}{A}{B}{C}{M}}
\put(130,0){\dvbc{D^X}{C^*}{B}{A^*}{M}}
\put(230,0){\dvbc{D^Y}{A}{C^*}{B^*}{M}}
\end{picture}
\caption{\label{fig:dvb}}
\end{figure}

There is a non-degenerate pairing between $D^X$ and $D^Y$ given, with the conventions of
\cite{Mackenzie:GT}, by
\begin{equation}                       
\label{eq:3duals}
\lhangle\Phi, \Psi\rhangle = 
\langle\Phi, d\rangle_A - \langle\Psi, d\rangle_B
\end{equation}
where $\Phi\in D^Y$ and  $\Psi\in D^X$ project to the same element of $C^*$ and $d$ is 
any element of $D$ which can be paired with $\Phi$ over $A$ and with $\Psi$ over $B$. 
This definition depends on the placement of $D$; the pairing for $D^f$ is the
negative of \eqref{eq:3duals}. 

The pairing induces an isomorphism of double vector bundles 
\begin{equation}
\label{eq:Q}
Q\co D^{XYX}\to D^f. 
\end{equation}
This isomorphism is natural in the usual informal sense of the word and preserves 
$A$ and $B$ but it induces minus the identity on the core. One can modify $Q$ so that 
it will induce minus the identity on one of the side bundles instead, but not so that 
it will induce the identity on all of $A$, $B$, and $C$. Accordingly, 
\cite{Mackenzie:2005dts} does not identify $D^{XYX}$ with $D^f$. At the end of this 
section we will show that although there does exist an isomorphism between $D^{XYX}$ and $D^f$ 
which induces the identity on $A$, $B$, and $C$ there is no canonical one.  This shows
that $XYX$ and $f$ are different as functors. 

In \cite{Mackenzie:2005dts} the group of dualization operations, there denoted 
$\EuScript{VB}_2$, is taken to be the group generated by $X$ and $Y$, subject to  $X^2 = I$ 
and $Y^2 = I$ and the identifications which follow from the existence of the non-degenerate 
pairing between the two duals. Here we denote this group by $\dg_2$, for \emph{duality
functors}, and we provide a formal definition for it in \ref{defin:dg2}.


\section{Duality functors} 
\label{sect:df}

In this section we will illustrate the methods which we will use in the triple and 
$n$-fold cases by showing that $XYX$ and $YXY$ are equal as functors, but that $XYX$ and 
$f$ are not. Indeed we will see that $f$ is not an element of the group of duality 
functors. 

\begin{defin}
\label{defin:stat}
Consider two \dvbts $\ts{D}$ and $\ts{D'}$ which have the same side bundles $A$ and $B$, 
and the same core bundle $C$. A \emph{statomorphism} $\phi\co D\to D'$ is a morphism
of \dvbts which preserves $A$, $B$ and $C$. 
\end{defin}

A statomorphism is necessarily an isomorphism $D\to D'$. All dualizations of a
statomorphism are also statomorphisms. 

The role played for vector spaces by the choice of a basis, or equivalently the
choice of an isomorphism $V\to\rr^n$, is played for multiple vector bundles by the 
notion of decomposition. First recall the notion of a decomposed \dvbt. 

Let $A$, $B$ and $C$ be vector bundles over $M$. Consider their fibered product
\begin{equation*}
\ABC = \{ (a,b,c)\in A \times B \times C \; \vert \; 
q^A(a) = q^B(b) = q^C(c) \in M \}. 
\end{equation*}
This has a pullback vector bundle structure over $A$:
\begin{equation*}
(a,b,c) \un{+}{A} (a,b',c') : = (a,b+b',c+c')
\end{equation*}
and a pullback vector bundle structure over $B$:
\begin{equation*}
(a,b,c) \un{+}{B} (a',b,c') : = (a+a',b,c+c')
\end{equation*}
With these structures $\ABC$ is called the \emph{decomposed \dvbt{}} with side bundles 
$A$ and $B$ and core bundle $C$ and denoted $A\bz B\bz C$. The structures on
$A\bz B\bz C$ are $q_A^!(B\oplus C)$ and $q_B^!(A\oplus C)$. These need to be
distinguished from the Whitney sum bundle $A\oplus B\oplus C$ on base $M$. 

In general, if $\ts{D}$ is a \dvbt with side bundles $A$ and $B$ and core bundle $C$, 
we say that $\ts{A\bz B\bz C}$ is the \emph{decomposed \dvbt associated to $\ts{D}$} and 
write $\ts{\decom{D}}: = \ts{A\bz B\bz C}.$

\begin{defin}  \label{defin:decom}
A \emph{decomposition} of $\ts{D}$ is a statomorphism $P\co\ts{D}\to\ts{\decom{D}}$. 

The set of decompositions of the \dvbt $\ts{D}$ will be denoted $\Dec(\ts{D})$.  
\end{defin}

\begin{rem}
Let $U$ be an open subset of $M$ such that the restrictions of the side bundles 
and the core over $U$ are trivializable.  If we compose a decomposition of $D$ 
restricted over $U$ with trivializations of $A$, $B$, and $C$, we obtain a statomorphism onto:
$$\dvbs{U \times V_A \times V_B \times V_C}{U \times V_A}{U \times V_B}{U}$$
where $V_A$, $V_B$, and $V_C$ are the fiber types of $A$, $B$, and $C$ respectively.  
This produces convenient local coordinates for the double vector bundle $D$.  Thanks 
to a \v{C}ech cohomology argument, the existence of local coordinates is equivalent 
to the existence of a decomposition (which, as explained below, is equivalent to the 
splitting condition).
\end{rem}

Given a decomposition $P$, there is a splitting $\Sigma(a,b) = P^{-1}(a,b,0)$ and conversely
given a splitting $\Sigma$, a decomposition is defined by $P^{-1}(a,b,c) = 
\Sigma(a,b) +_A (c +_B\tilo_a).$

Denote by $\Stat(A\bz B \bz C)$ the group of statomorphisms from $A\bz B\bz C$ to itself;
that is, $\Stat(A\bz B \bz C) = \Dec(A \bz B \bz C)$.  
Each element of $\Stat(A\bz B\bz C)$ is of the form
\begin{equation}  \label{eq:philambda}
\phi_\lambda(a,b,c) = (a,b,c+\lambda(a, b))
\end{equation}
where $\lambda\co A\otimes B\to C$ is a linear map. (Throughout the paper we will usually
write tuples in place of tensor elements, when they are the argument of a linear map.) 
Accordingly $\Stat(A\bz B\bz C)$ is a group with composition 
$\phi_\lambda\circ\phi_{\lambda'} = \phi_{\lambda+\lambda'}$
and can be identified with the additive group $\Ga(A^*\otimes B^*\otimes C)$. (Here 
and elsewhere in the paper, the tensor products are over $M$.) We do not
consider the $\cinf{M}$-module structure on $\Ga(A^*\otimes B^*\otimes C)$.
\label{comment:affine}

The duals of $A\bz B\bz C$ are the decomposed \dvbts $C^*\bz B\bz A^*$ and $A\bz C^*\bz B^*$. 
Again, the statomorphisms from $C^*\bz B\bz A^*$ to itself form an abelian group
isomorphic to $\Ga(C\otimes B^*\otimes A^*)$, and those from $A\bz C^*\bz B^*$ to itself 
form an abelian group $\Ga(A^*\otimes C\otimes B^*)$. In what follows we will identify
these three groups without comment and denote them by $G_2$. 

Let $\ts{D}$ be a \dvbt with side bundles $A$ and $B$ and core bundle $C$. Then 
$G_2 = \Stat(A\bz B\bz C)$ acts on $\Dec(\ts{D})$ to the left, and the action is simple and
transitive. Thus $\Dec(\ts{D})$ is a (non--empty) $G_2$--torsor. 

We now wish to consider the effect of the dualization operations on morphisms between
\dvbts, and specifically on decompositions. We first define the relevant 
category. 

\begin{defin}
Let $\mathcal{C}$ be the category whose objects are double vector bundles, and whose 
morphisms are statomorphisms of double vector bundles.
\end{defin}

For the rest of this section, fix three vector bundles $A$, $B$ and $C$ over $M$ 
and, for the moment, denote them by $E_1:=A$, $E_2 := B$ and $E_0 :=C\du$.  (In the
triple case we will use an extension of this notation exclusively.) 
Let $S_3$ be the group of permutations of the set $\{0,1,2\}$.
We will also write $\C^{op}$ or $\C^{-1}$ for the opposite category to $\C$. 

We can now think of $X$, $Y$ and $f$ as functors:
\begin{equation}
\label{eq:XYf}
\vb{\mscr{C}}{\mscr{C}^{op}}{X},\qquad
\vb{\mscr{C}}{\mscr{C}^{op}}{Y},\qquad
\vb{\mscr{C}}{\mscr{C}}{f}
\end{equation}
The effect of $X$ on objects has already been defined
and the action on morphisms is as expected~: 
Let $\phi\co\ts{D_1} \Lra \ts{D_2}$ be a morphism in the category $\mscr{C}$.  
Then $\phi$ is a statomorphism of \dvbts. In particular
$\phi\co D_1\to D_2$ is a morphism of vector bundles over $B$ and so can be
dualized to $\phi^X\co D_2^X\to D_1^X$; this is also a statomorphism of \dvbts
$\ts{D_2^X} \Lra \ts{D_1^X}$ and is thus a morphism 
$\ts{\phi^X}\co \ts{D_1^X} \Lra \ts{D_2^X}$ in the category
$\mscr{C}^{op}$.

Note that, since every morphism in $\C$ is invertible, we can regard $X$ and $Y$ 
as functors $\C \to \C^{op}$ or as functors $\C^{op} \to \C$; hence, we can compose 
them, and we can talk of the group of functors generated by $X$ and $Y$. Denote this
group by $\W_2$. Now for any $W \in \W_2$ there exists a unique permutation $\sigma \in S_3$ 
such that, if $D$ is a double vector bundle with side bundles $E_1$ and $E_2$, and with core 
bundle $E_0^{\star}$, then $D^W$ is a double vector bundle with side bundles 
$E_{\sigma (1)}$ and $E_{\sigma(2)}$, and with core bundle $E_{\sigma(0)}^{\star}$.  
Define $\pi(W) := \sigma$,  and define $\varepsilon_W = \pm1 $ to be the signature 
of $\pi(W)$. Now $\vb{\mscr{C}}{\mscr{C}^{ \varepsilon_W}}{W}$ is a functor.

The flip operation $f$ is most naturally considered as a covariant functor; it can be 
defined as such for all morphisms of \dvbts, not only for statomorphisms. 

Finally, we recall the definition of natural isomorphism.  

\begin{defin}
Let ${\tt Cat}_i$, $i=1,2$, be two categories, and let 
$F, G \co {\tt Cat}_1 \Lra {\tt Cat}_2$ be two functors.  
A \emph{natural transformation}  $s\co F \Lra G$ is a collection of morphisms 
$s(O) \co F(O) \Lra G(O)$ in ${\tt Cat}_2$ for every object $O$ in ${\tt Cat}_1$ 
such that, given any morphism $f\co O \Lra O'$ in ${\tt Cat}_1$, the following diagram 
is commutative:
\begin{equation*}
\dvb{F(O)}{G(O)}{F(O')}{G(O')}{s(O)}{F(f)}{G(f)}{s(O')}
\end{equation*}
If, in addition, $s(O)$ is an isomorphism in ${\tt Cat}_2$ for every object 
$O$ in ${\tt Cat}_1$, the natural transformation $s$ is called a \emph{natural isomorphism}.
\end{defin}

Consider $W \in \W_2$ such that $\pi(W)= 1 \in S_3$. Thus $W$ is a word in $X$ and $Y$, 
and as a functor, $W$ is $\mscr{C}\to\mscr{C}.$ We want criteria for a natural isomorphism
to exist between $W$ and the identity functor. We resume the notation of Figure~\ref{fig:dvb}
so as to facilitate relating this section to \cite{Mackenzie:GT}. 

Let $\ts{D}$ be a \dvbt with side bundles $A$ and $B$ and core $C$. 
Since $\pi(W) = 1$, $\ts{D^W}$ is also a \dvbt with side bundles $A$ and $B$ and 
core $C$.  Choose a decomposition $P\co \ts{D} \to \ts{\decom{D}}$ and apply the functor 
$W$ to it. We obtain a decomposition of $D^W$:
$$
P^W \co \ts{D^W} \Lra \decom{\ts{D^W}} = \ts{{\decom{D}}^W} = \decom{D} =\ABC. 
$$
Hence we obtain a statomorphism $(P^W)^{-1} \circ P\co \ts{D}\to\ts{D^W}$. 
This is our candidate for a natural isomorphism between the identity functor and the 
functor $W$. In order for it to succeed, it should at least be independent of the choice 
of decomposition. The next two theorems together show that this is also a sufficient 
condition. 

The following result is valid for any element of $\W_2$. 
Part (i) has been stated already and is included here for reference. 

\renewcommand{\theenumi}{{\rm (\roman{enumi})}}

\begin{thm} \label{thm:submain}  
Let $W \in \W_2$ and let $\ts{D}$ be a \dvbt with side bundles $A$ and $B$, and core $C$. 
\begin{enumerate}
\item The spaces of decompositions $\Dec(\ts{D})$ and $\Dec(\ts{D^W})$ are $G_2$-torsors. 

\item  Choose a decomposition $P_0 \in \Dec(\ts{D})$.  Then the bijection 
$\vartheta_W \co \Dec(\ts{D}) \to \Dec(\ts{D^W})$ defined by 
$$
\vartheta_W(P) = (P^W)^{\varepsilon_W}, 
$$
induces a group automorphism $\theta_W \co G_2 \to G_2$ such that
$$ 
\vartheta_W(\phi \circ P_0) = \theta_W(\phi) \circ \vartheta_W(P_0)
$$  
for $\phi\in G_2$. 
Moreover, $\theta_W$ does not depend on the choice of $P_0$ (that is, $\vartheta_W$ is a 
morphism of $G_2$-torsors) nor on the choice of the double vector bundle $D$.

\item  The map $\W_2 \times \Wii\to G_2,\quad (W,\lambda) \mapsto \theta_W(\lambda),$
is a group action.  
\end{enumerate}
\end{thm}

\pf
(ii)  It is enough to prove the result for $W = X$, since the same proof will hold 
for $Y$, and then for any word $W$ in $X$ and $Y$.

Consider $P \mapsto (P^X)^{-1},\ \Dec (\ts{D}) \to \Dec (\ts{D^X})$. For
$\phi\in\Stat(A\bz B\bz C)$ we have, by functoriality, 
$$
\vartheta_X(\phi\circ P) = (\phi^X)^{-1}\circ (P^X)^{-1} = (\phi^X)^{-1}\circ \vartheta_X(P),
$$
so that $\theta_X(\phi) = (\phi^X)^{-1}$.

Let us write $\phi = \phi_{\lambda}$ for some $\lambda \co A \otimes B \to C$ as in \eqref{eq:philambda}.  
Similarly, let us write $\theta_X(\phi) = \phi_{\lambda'}$ for some $\lambda' \co C^* \otimes B \to A^*$.  
Notice that with our abuse of notation we can also write $\theta_W(\lambda) = \lambda'$.
Consider the following diagram:
\begin{equation*}
\xymatrix{D \ar[r]^{P} \ar[rd]_{\phi_{\lambda}\circ P} & \decom{D} \ar[d]^{\phi_{\lambda}} & \decom{D^X} \ar[r]^{P^X} \ar[d]_{\phi_{\lambda'}} & D^X  \\
& \decom{D} & \decom{D^X} \ar[ur]_{(\phi_{\lambda} \circ P)^X}   }
\end{equation*}

If $d \in D_b$ and $\delta \in D^X_b = (D \duer B)_b$, we can calculate the pairing of $d$ and $\delta$ 
in any decomposition.  That is:
\begin{equation*}
\langle \delta \vert d \rangle \; = \; 
\langle  (P^X)^{-1}(\delta)       \vert  P(d)  \rangle \; = \;
\langle  \phi_{\lambda'} \circ (P^X)^{-1}(\delta) \vert  (\phi_{\lambda} \circ P)(d) \rangle
\end{equation*}
Let us write $P(d) = (a,b,c) \in \ABC$ and $(P^X)^{-1}(\delta)= (\gamma,b,\alpha) \in C^* \bz B \bz A^*$.  Then:
$$
\langle (P^X)^{-1}(\delta) \vert P(d) \rangle = 
\langle (\gamma,b,\alpha) \vert (a,b,c) \rangle =  
\langle \alpha \vert a \rangle + \langle \gamma \vert c \rangle 
$$
and 
\begin{align*}
\langle \phi_{\lambda'} ((P^X)^{-1}(\delta)) \vert  \phi_{\lambda} (P(d)) \rangle  = &
\langle (\gamma',b,\alpha + \lambda(\gamma,b)) \vert (a,b,c+ \lambda(a,b))\rangle\\
 = & \langle \alpha \vert a \rangle + \langle \gamma \vert c \rangle +
\lambda(a,b,\gamma) + \lambda'(a,b,\gamma).
\end{align*}
In order for these to be equal we need $\lambda = - \lambda'$.  We have proved that $\theta_X$ is 
minus the identity on $G_2$, and that $(\phi)^X = \phi$. 
This completes the proof of (ii). 

\medskip

(iii) is clear now that we know the map $\theta$ is well defined.
\pfend

The action of $\W_2$ on $G_2$ is not faithful.  However, if we restrict it to the kernel 
of $\pi \co\W_2\to S_3$ then it becomes faithful, as the next theorem shows.

\begin{thm} \label{thm:main}
For $W \in \W_2$ with $\pi(W) = 1$, the following are equivalent: 
\begin{enumerate}
\item $\theta_W$ is the identity on $\Wii$. 
\item There exists a natural isomorphism between the identity functor and $W$.
\end{enumerate}
\end{thm}

\pf
$\textrm{(i)} \implies \textrm{(ii)}$ For every \dvbt $D$ in $\mscr{C}$ we want to 
obtain a statomorphism $s(D) \co D \to D^W.$ Choose a decomposition $P \co D \to A\bz B\bz C$.  
We also have $P^W \co D^W \to A\bz B\bz C$.  Compose them to define $s(D):=(P^W)^{-1} \circ P$.

We check first that $s(D)$ defined in this way does not depend on the choice of $P$. 
If $P_1, P_2 \in \Dec(\ts{D})$, there is some $\lambda \in \Wii$ such that 
$P_2 = \phi_{\lambda} \circ P_1$.  We then have 
$(P_2)^W = \phi_{\theta_W(\lambda)} \circ (P_1)^W$.  
Since $\theta_W = \id_{G_2}$, it follows that 
$\phi_{\theta_W(\lambda)}=\phi_{\lambda}$, and the following diagram is commutative:
\begin{equation*}
\xymatrix{
D \ar[r]^(.3){P_1} \ar[rd]_{P_2} & \ABC \ar[d]^{\phi_{\lambda}} \ar@{}[r]^(.7){(P_1)^W} &
D^W \ar[l] \ar[ld] \\
& \ABC \ar@{}[ur]_{(P_2)^W}
}
\end{equation*}
which proves that $s(D)$ is well defined. 

We need to check that $s$ is a natural transformation.  Let $f\co D_1 \to D_2$ be a 
statomorphism of \dvbts. Choose a decomposition $P_2\co D_2 \to \ABC$ of $D_2$ and 
consider the decomposition $P_1 :=P_2 \circ f$ of $D_1$.
Then the following diagram is commutative:
\begin{equation*}
\xymatrix{
\ts{D_1} \ar[r]_(.3){P_1}  \ar@/^1pc/[rr]^{s(D_1)} \ar[d]_{f} &
\ts{\ABC} \ar@{}[r]_(.7){(P_1)^W} \ar[d]^{=} &
\ts{(D_1)^W} \ar[l] \ar[d]^{f^W} \\
\ts{D_2} \ar[r]^(.3){P_2} \ar@/_1pc/[rr]_{s(D_2)} &
\ts{\ABC} \ar@{}[r]^(.7){(P_2)^W} & 
\ts{(D_2)^W} \ar[l]
}
\end{equation*}

$\textrm{(ii)} \implies \textrm{(i)}$  Reciprocally, let $s$ be a natural 
transformation from the identity functor to $W$.  Let $\ts{D}$ be a \dvbt in $\mscr{C}$ 
and let $\decom{\ts{D}} = \decom{\ts{D^W}} = {\decom{\ts{D}}}^W$ be the 
decomposed \dvbt.  Let $P_1, P_2 \in \Dec(\ts{D})$ be two decompositions of $\ts{D}$ 
such that $P_2=\phi_{\lambda} \circ P_1$ for $\lambda \in \Wii$. Then the following 
diagram is commutative:
\begin{equation*}
\xymatrix{
\ts{D} \ar[r]_{P_1} \ar[d]_{=} \ar@/^1pc/[rrr]^{s(D)} &
\ts{\Bar{D}} \ar[r]_{s(\Bar{D})} \ar[d]_{\phi_{\lambda}} & 
\ts{\Bar{D}} \ar@{}[r]_{(P_1)^W} \ar[d]_{(\phi_{\lambda})^W} &
\ts{D} \ar[l] \ar[d]_{=} \\
\ts{D} \ar[r]_{P_2} & 
\ts{\Bar{D}} \ar[r]_{s(\Bar{D})} &
\ts{\Bar{D}} \ar@{}[r]_{(P_2)^W} &
\ts{D} \ar[l]
}
\end{equation*}
Hence $\phi_{\theta_W(\lambda)} =(\phi_{\lambda})^W = \phi_{\lambda}$
and $\theta_W(\lambda) = \lambda$.
\pfend

We can now give a precise definition of $\dg_2$ and calculate it. 

\begin{defin}
\label{defin:dg2}
The group $\dg_2$ is the group $\W_2$ of functors generated by $X$ and $Y$, quotiented over
natural isomorphism in $\mscr{C}$. 
\end{defin}

From (\ref{eq:XYf}) we have $\sigma_1: = \pi(X) = (01)$ and $\sigma_2 : = \pi(Y) = (02)$ 
and it follows that $\pi \co \dg_2\to S_3$ is surjective. Define $K_3$ to be the 
kernel of $\pi$. The following lemma is an easy exercise in algebra. 

\begin{lem} \label{lem:gener}
Let $f\co G \to S$ be a surjective group homomorphism.  Let $g_1,\ldots,g_n$ be a set of 
generators of $G$ and write $\sigma_i:=f(g_i)$.  Let 
$\{ R_j(\sigma_1,\ldots,\sigma_n) \; \vert \; j=1,\ldots,m \}$ be a set of relations for 
a presentation of $S$ with generators $\sigma_1,\ldots,\sigma_n$. 
Then the kernel of $f$ is the normal subgroup of $G$ generated by 
$\{ R_j(g_1,\ldots,g_n) \; \vert \; j=1,\ldots,m  \}$.
\end{lem}

A presentation of $S_3$ with generators $\sigma_1, \sigma_2$ is
\begin{equation*}
\langle  \sigma_1, \sigma_2 \; \vert \; \sigma_1^2,\, \sigma_2^2,\, (\sigma_1\sigma_2)^3 \rangle .
\end{equation*}
Hence, $K_3 = \ker \pi$ is the normal subgroup of $\dg_2$ generated by 
\begin{equation*}
\{ X^2,\, Y^2,\, (XY)^3 \} .
\end{equation*}

In the proof of Theorem \ref{thm:submain} we showed that 
$\theta_X = \theta_Y = - \id.$ From this we obtain
\begin{equation*}
\theta_{X^2} = \theta_{Y^2} = \theta_{(XY)^3} = \id. 
\end{equation*}
By Theorem \ref{thm:main} it follows that $X^2 = Y^2 = (XY)^3 = 1$ in $\dg_2$. 
Thus $K_3$ is the trivial group and $\dg_2 = S_3$, as in \cite{Mackenzie:2005dts}. 
In the following sections we will extend this method to determine the group $\dg_3$. 

We end the present section by demonstrating that, for any \dvbt $D$, a choice of 
statomorphism between $D^{XYX}$ and $D^f$ is equivalent to choosing a decomposition 
of $D$. 

First, let $P\co \ts{D} \Lra \ts{\ABC}$ be a decomposition. Applying the functors $XYX$ 
and $f$ we get statomorphisms
$$
P^{XYX}\co  B\bz A\bz C  \Lra \ts{D^{XYX}}\quad\text{and}\quad
P^{f}\co \ts{D^f} \Lra  B\bz A\bz C. 
$$
Hence $\Phi(S) := P^{XYX} \circ P^f \co \ts{D^f} \Lra \ts{D^{XYX}}$ is a statomorphism.  
The interesting result is that those are all the statomorphisms.

\begin{prop}
Let $\mathcal{S}$ denote the set of all statomorphisms $\ts{D^f} \Lra \ts{D^{XYX}}$. The map
$$
\Phi\co  \Dec(\ts{D}) \Lra \mathcal{S}
$$
defined above is a bijection.
\end{prop}

\begin{proof}
Let $h\co \ts{D^f} \Lra \ts{D^{XYX}}$ be a statomorphism.  We want to prove that there is 
a unique $P \in \Dec(\ts{D})$ such that $\Phi(P) = h$.  First, 
using notation similar to that in the proof of Theorem \ref{thm:submain}, we note that 
the groups of statomorphisms of the decomposed \dvbts $\ts{\ABC}$ and 
$B \bz A \bz C$ are canonically isomorphic to $G_2$. 
\begin{equation*}
\begin{split}
\Tilde{\phi}: \; \Wii \Lra & \Stat(B \bz A \bz C) \\
\lambda \; \; \mapsto & \; \Tilde{\phi}_{\lambda}: B \bz A \bz C \Lra B \bz A \bz C \\
            & \quad \quad (b,a,c) \mapsto (b,a,c + \lambda(a,b))
\end{split}
\end{equation*}
We also know from the proof of Theorem \ref{thm:submain}(ii) that 
$(\phi_{\lambda})^f = (\phi_{\lambda})^{XYX}=\Tilde{\phi}_{\lambda}$. 

Let us pick a decomposition $P_1\in \Dec(\ts{D})$. Then 
$((P_1)^f)^{-1} \circ h \circ ((P_1)^{XYX})^{-1}$ is an automorphism of 
$B \bz A \bz C$, and hence equals $\Tilde{\phi}_{\mu}$ for some $\mu \in \Wii$.
\begin{equation*}
\xymatrix{
\ts{D}^f \ar[r]^{h} \ar[d]_{(P_1)^f} & \ts{D}^{XYX} \ar@{}[d]^{(P_1)^{XYX}} \\
\ts{B \bz A \bz C} \ar[r]_{\Tilde{\phi}_{\mu}} & \ts{B \bz A \bz C} \ar[u] 
}
\end{equation*}
We know that every decomposition of $\ts{D}$ is of the form $\phi_{\lambda} \circ P_1$ for a 
unique $\lambda \in \Wii$.  So we want to prove that there is a unique $\lambda \in \Wii$ 
such that $h = \Phi(\phi_{\lambda} \circ P_1)$.  Noting that
$$
h  = (P_1)^{XYX} \circ \Tilde{\phi}_{\mu} \circ (P_1)^f 
$$
we have that
\begin{multline*}
\Phi(\phi_{\lambda} \circ P_1) = 
(\phi_{\lambda} \circ S_1)^{XYX} \circ (\phi_{\lambda} \circ P_1)^f = 
(P_1)^{XYX} \circ (\phi_{\lambda})^{XYX} \circ (\phi_{\lambda})^f \circ (P_1)^f  \\
= (P_1)^{XYX} \circ \Tilde{\phi}_{\lambda} \circ \Tilde{\phi}_{\lambda} \circ (P_1)^f = 
(P_1)^{XYX} \circ \Tilde{\phi}_{2 \lambda} \circ (P_1)^f
\end{multline*}
and so $h = \Phi(\phi_{\lambda} \circ P_1)$ if and only if $2 \lambda = \mu.$
This completes the proof.
\end{proof}

To summarize: there is a canonical isomorphism between $D^{XYX}$ and $D^f$, but it is not 
a statomorphism;  there are statomorphisms between $D^{XYX}$ and $D^f$, but not canonical 
ones, as choosing one such statomorphism is equivalent to choosing a decomposition; the 
functors $XYX$ and $f$ are not naturally isomorphic through statomorphisms.  Because of 
this, we regard $XYX$ and $f$ as distinct  functors.

\begin{ej}
\label{ej:conns}
It is worthwhile to consider the significance of decompositions for the tangent
prolongation of an ordinary vector bundle. 

Let $\vb{A}{M}{q}$ be a vector bundle. Applying the tangent functor, we obtain a \dvbt:
\begin{equation*}
\dvb{TA}{A}{TM}{M}{\pi_A}{Tq}{q}{\pi_M}
\end{equation*}
with side bundles  $A$, and $TM$. The core consists of those $X\in TA$
which are annulled by both $T(q)$ and $\pi_A$; that is, $C$ is the set of vertical
vectors along the zero section, and may be identified with $A$. 

A splitting of $TA$ is a connection in $A$ in the usual sense; see, for example, 
\cite{Besse}. Given a splitting 
\begin{equation*}
\Sigma\co TM\times_M A\to TA,
\end{equation*}
and a vector field $x$ on $M$, define a vector field $\Tilde{x}$ on $A$ by
$$
\Tilde{x}(e) = \Sigma(x(q(e)),e).
$$
It follows from the properties of $\Sigma$  that $x\mapsto\Tilde{x}$ 
is a connection in $A$. Conversely a connection in $A$ induces a decomposition, 
and $\Dec(TA)$ can be identified with the space of connections in $A$. 

As is well known, connections in $A$ form an affine space with model space 
$\Gamma(T^* M \otimes\End(A))$. This includes the $G_2$-torsor structure. 

The dual over $TM$, namely $(TA)^X$, can be identified with the tangent
double vector bundle $T(A^*)$ \cite[9.3.2]{Mackenzie:GT} and so 
Theorem \ref{thm:submain}(ii) includes the correspondence between connections
in $A$ and connections in $A^*$. 
\end{ej}


\section{Triple vector bundles} 
\label{sect:tvb}

We begin by recalling the definition of a triple vector bundle. See
\cite{Mackenzie:2005dts} for fuller details. We use a notation that extends readily 
to the $n$-fold case. 

\begin{figure}[h]
\setlength{\unitlength}{1cm}
\begin{picture}(15,5)(-0.5,-0.5)
\put(0,4)
{\smallcube{E_{1,2,3}}{E_{2,3}}{E_{1,3}}{E_3}{E_{1,2}}{E_2}{E_1}{M}}
\put(7,3.6){\xymatrix@=2mm{
&& E_{3,12}\ar[dddd] &&\\
& E_{2,31}\ar[rrdd] &&&\\
E_{1,23}\ar[rrrr] &&&& E_{23}\\
&&& E_{13} &\\
&& E_{12} &&\\
}}
\end{picture}
\caption{\ \label{fig:tvb}}
\end{figure}

\begin{defin}
A \emph{triple vector bundle} $E$ is a system as on the left of Figure~\ref{fig:tvb}
where each arrow represents a vector bundle structure, such that each face
is a double vector bundle, and which satisfies the splitting condition stated
below. 

A \emph{morphism of triple vector bundles} $\phi\co E\to F$ is a system of maps
$\phi_I\co E_I\to F_I$, for all subsets $I$ of $\{1,2,3\}$, such that for all
nonempty subsets $I$ and each $k\in I$, $\phi_I$ is a morphism of vector
bundles over $\phi_{I\backslash\{k\}}$. 
\end{defin}

We always read figures as in Figure~\ref{fig:tvb} with 
$(E_{1,2,3};E_{1,2}, E_{2,3};E_2)$ at the rear and $(E_{1,3};E_1,E_3;M)$ coming 
out of the page toward the reader. Notice that when a nonempty subscript has $k$ 
commas, $0\leq k\leq 2$, the space which it denotes has $k+1$ vector bundle 
structures. We sometimes denote $M$ by $E_\emptyset$ for uniformity. 

The three structures of double vector bundle on $E_{1,2,3}$ are the \emph{upper} 
double vector bundles, and $E_{1,2}, E_{2,3}, E_{3,1}$ are the \emph{lower} double 
vector bundles. We refer to $(E_{1,2};E_1, E_2;M)$ as the \emph{floor} of
$E$ and to $(E_{1,2,3};E_{1,3}, E_{2,3};E_3)$ as the \emph{roof} of $E$, 
with \emph{left, right, front} and \emph{back} having their usual meanings. 

The cores of the lower double vector bundles are denoted 
$E_{12},\, E_{23},\, E_{31},$ and the cores of the upper double 
vector bundles are denoted $E_{3,12},\,E_{1,23},\,E_{2,31}.$
The latter are vector bundles over the former, as indicated on the 
right of Figure~\ref{fig:tvb}, and  form three \emph{core double vector 
bundles}, $(E_{1,23}; E_1,E_{23};M)$, $(E_{2,31}; E_2,E_{31};M)$, and
$(E_{3,12}; E_3,E_{12};M)$, Each of the three core double vector bundles 
has the same core, denoted $E_{123}$ and called the \emph{ultracore}. 
The seven vector bundles $E_1,\,E_2,\,E_3,\,E_{12},\,E_{23}$, $E_{31},\,E_{123}$ 
are called the \emph{building bundles of} $E.$ We denote them collectively
by $E_\sol$.

\begin{ej}
The tangent prolongation of a double vector bundle $(D;A,B;M)$ is the triple
vector bundle $TD$ shown in Figure~\ref{fig:tvbexs}(a). The three core double
vector bundles are $(D;A,B;M)$, $(D;A,B;M)$ and $(TC;C.TM;M)$, where $C$ is the
core of $D$. The ultracore is $C$. 

\begin{figure}[h]
\begin{center}
\subfloat[]
{\smallcube{TD}{TB}{TA}{TM}{D}{B}{A}{M}}
\qquad
\subfloat[]
{\xymatrix@=5mm{
                             & & E_{2,3}\ar'[d][dd] \ar[dr] &\\
& E_{1,3}   \ar[rr] \ar[dd]    & & E_3 \ar[dd]  \\
E_{1,2}\ar'[r][rr] \ar[dr]  &       & E_2 \ar[dr] &\\
& E_1  \ar[rr] & & M\\
}}
\caption{\ \label{fig:tvbexs}}
\end{center}
\end{figure}

Suppose given three double vector bundles arranged as in Figure~\ref{fig:tvbexs}(b)
and a vector bundle $E_{123}$ on $M$. 
Define the manifold $E'_{1,2,3}$ to be the pullback of the diagram; that is, $E'_{1,2,3}$ consists of
triples $(e_{1,2},\,e_{1,3},\,e_{2,3})\in E_{1,2}\times E_{1,3}\times E_{2,3}$ such that
$$
q^{1,2}_1(e_{1,2}) = q_1^{1,3}(e_{1,3}),\quad
q^{1,2}_2(e_{1,2}) = q_2^{2,3}(e_{2,3}),\quad
q^{1,3}_1(e_{1,3}) = q_3^{1,3}(e_{1,3}),
$$
using an obvious notation for the projections $q_k^{i,j}$. Finally, define
$E_{1,2,3}$ to be the manifold of tuples 
$(e_{1,2},\,e_{1,3},\,e_{2,3},\,e_{123})\in E'_{1,2,3}\times E_{123}$ such that 
$(e_{1,2},\,e_{1,3},\,e_{2,3})$ projects to the same point of $M$ as does $e_{123}$. 
Define a vector bundle structure on $E_{1,2,3}$ with base $E_{1,2}$ by
\begin{multline*}
(e_{1,2},\,e_{1,3},\,e_{2,3},\,e_{123}) \add{1,2} (e_{1,2},\,e'_{1,3},\,e'_{2,3},\,e'_{123})\\ 
= (e_{1,2},\,e_{1,3}\add{1}e'_{1,3},\,e_{2,3}\add{2}e'_{2,3},\,e_{123}+e'_{123}), 
\end{multline*}
and scalar multiplication analogously. With the similar structures over $E_{1,3}$ and
$E_{2,3}$, this makes $E_{1,2,3}$ a triple vector bundle with ultracore $E_{123}.$ 
Thus any diagram of the form Figure~\ref{fig:tvbexs}(b) can be completed to a triple
vector bundle with an arbitrary ultracore. 

Given an indexed set of vector bundles 
$E_\sol = \{E_1,\,E_2,\,E_3,\,E_{12},\,E_{23},\,E_{31},\,E_{123}\}$ over base $M$, performing
this construction with the decomposed double vector bundles $E_{i,j} = E_i\bz E_j\bz E_{ij}$
yields the \emph{decomposed triple vector bundle} $\Bar{\BE}$ with the $E_\sol$ as building 
bundles.  (Note that the bar here means something slightly different from the notation 
$\decom{D}$ for a double vector bundle $D$.  In both cases the bar means that we build a 
decomposed $n$-fold vector bundle, but in one case we start with a general $n$-fold vector 
bundle, while in the other we start with a set of building bundles.)
\end{ej}

\begin{defin}
\label{df:stato3}
Let $E$ and $F$ be triple vector bundles with the same building bundles
$E_\bullet$. A \emph{statomorphism} from $E$ to $F$ is an isomorphism
$\phi\co E\to F$ which induces the identity on each of the building bundles. 

A triple vector bundle $E$ satisfies the \emph{splitting condition} if there is a
statomorphism of triple vector bundles $\phi\co E\to \Bar{\BE}$; 
such a map is a \emph{splitting of} $E$. 
\end{defin}

\begin{prop}
If a double vector bundle $(D;A,B;M)$ satisfies the splitting condition, then its
tangent prolongation $TD$ does also.   
\end{prop}

\pf
Taking the tangent of a decomposition of $D$ we have a diffeomorphism 
$TD\to TA\times_{TM} TB\times_{TM} TC$. Using decompositions of $TA, TB$ and $TC$, this
gives a diffeomorphism 
$$
TD\cong (A\bz A\bz TM)\times_{TM} (B\bz B\bz TM) \times_{TM} (C\bz C\bz TM) 
\cong A\bz A\bz B\bz B\bz C\bz C\bz TM,
$$
where $\bz$ denotes pullback over $M$. Setting $E_1 = A,\ E_2 = B,\ E_3 = TM,\ E_{12} = C,$
$E_{13} = A,\,E_{23} = B,\, E_{123} = C$ we have a decomposition of $TD$. 
\pfend

The duals of a triple vector bundle are set out in \cite[\S7]{Mackenzie:2005dts}; see 
Figure~\ref{fig:tvbXYZ}. Here we need a different notation to use for the duals of decomposed
triple vector bundles. For any subset $I\subseteq\{0,1,2,3\}$, for the duals of building
bundles, write
$$
E_I^* = E_{I^c},
$$
where $I^c$ us the complement of $I$. 
In particular $E_0 = E_{123}^*$ is the dual of the ultracore. 
The effect of dualization on the building bundles is shown in Table~\ref{table:tvbXYZ}. In 
this table the middle column shows the duals of the ultracores, rather than the ultracores themselves. 

\begin{figure}[h]
\setlength{\unitlength}{10mm}
\begin{picture}(15,6)(1,0)
\put(0,5){\xymatrix@!=1pc{
\save[]+<-1.5cm,0.1cm>*\txt<8pc>{$E^X = {}$} \restore
{E_{1,2,3}\duer E_{2,3}} \ar[rr] \ar[dd] \ar[dr] & & E_{2,3} \ar'[d][dd] \ar[dr] &\\
& E_{3,12}\duer E_3 \ar[rr] \ar[dd]         & & E_3 \ar[dd]  \\
E_{2,31}\duer E_2  \ar'[r][rr] \ar[dr]         & & E_2 \ar[dr] &\\
& E_0  \ar[rr] & & M\\
}}
\put(8.5,5){\xymatrix@!=0.4pc{
&& E_{1,3}\duer E_3\ar[dddd] &&\\
& E_{1,2}\duer E_2\ar[rrdd] &&&\\
E_{1,23}\duer E_{23}\ar[rrrr] &&&& E_{23}\\
&&& E_{03} &\\
&& E_{02} &&\\
}}
\end{picture}
\caption{The $X$ dual of a general triple vector bundle.\label{fig:tvbXYZ}}
\end{figure}
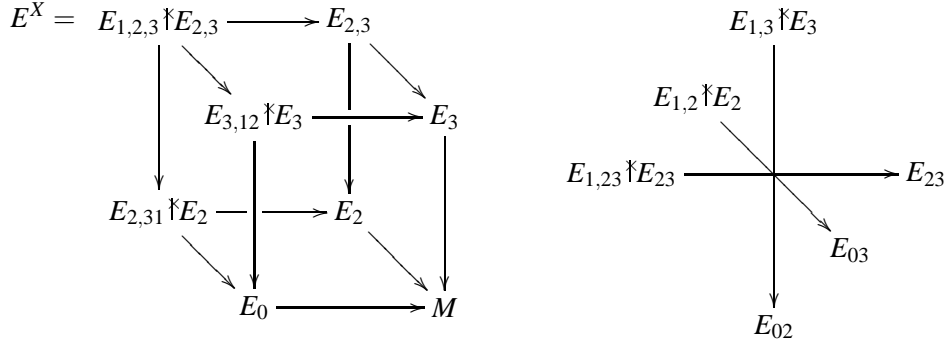

Define $\dg_3$ to be the group of functors generated by $X, Y, Z$ up to natural 
isomorphism, that is, two functors are considered the same if there is
a natural isomorphism between them through statomorphisms. (We will normally 
denote elements of $\dg_3$ by representatives.) As in the double case, 
there is a group homomorphism $\pi \co\dg_3\to S_4$, defined by the action of 
$X$, $Y$ and $Z$ on $E_1$, $E_2$, $E_3$ and $E_0$. We have
$$
\pi(X) = (01),\quad
\pi(Y) = (02),\quad
\pi(Z) = (03). 
$$

\newcolumntype{Z}{>{$}c<{$}}
\newcolumntype{L}{>{$}l<{$}}
\newcolumntype{R}{>{$}r<{$}}

\begin{table}[h] 
\begin{center}
\begin{tabular}{||L||R|R|R|R|R|R|R||}
\hline\hline
&E_1&E_2&E_3&E_0&E_{12}&E_{23}&E_{31}\\
\hline\hline
X &E_0&E_2&E_3&E_1&E_{02}&E_{23}&E_{03}\\\hline
Y &E_1&E_0&E_3&E_2&E_{01}&E_{03}&E_{31}\\\hline
Z &E_1&E_2&E_0&E_3&E_{12}&E_{02}&E_{01}\\\hline
\hline
\end{tabular}
\end{center}
\caption{The action on the building bundles, with the ultracores
replaced by their duals. \label{table:tvbXYZ}}
\end{table}

Write $\sigma_1, \sigma_2, \sigma_3$ for the images of $X$, $Y$ and $Z$. These generate 
$S_4$ and so $\pi$ is surjective. Write $K_4$ for the kernel of $\pi$. 
A presentation for $S_4$ in terms of these generators is 
\cite{CoxeterM:1965}: 
$$
\langle \sigma_1,\,\sigma_2,\,\sigma_3 \ |\ \sigma_i^2,\ (\sigma_i\sigma_j)^3,\ 
(\sigma_i\sigma_{j}\sigma_i\sigma_k)^2\rangle,
$$
where $i$, $j$, and $k$ are distinct.
Hence, by Lemma \ref{lem:gener}, $K_4$ is the normal subgroup generated by 
the 
elements in the list  
\begin{equation}
\label{eq:k4gens}
X^2,\,Y^2,\,Z^2,\,(XY)^3,\,(YZ)^3,\,(ZX)^3,\, (XYXZ)^2,\,(YZYX)^2,\,(ZXZY)^2.
\end{equation}

A priori, and according to our presentation, the list should contain further elements, but 
they are not needed since $XYX = YXY$ implies $(XYXZ)^2 = (YXYZ)^2$ and similar conditions.

From Table~\ref{table:tvbXYZ} it is clear that the action of $X,Y,Z$ on the set of
building bundles is determined by the action on the set $\{0,1,2,3\}$ of indices, and so acts
as $S_4$.   In particular, $W$ acts trivially on the set of building bundles if and only if 
$\pi(W)=1$.   As a consequence, since two triple vector bundles are statomorphic if and only 
if they have the same building bundles, we conclude that $W \in K_4$ if and only if $E$ and 
$E^W$ are statomorphic for every triple vector bundle $E$ (albeit not canonically, unless 
$W=1$). 

\begin{ej}
\label{ej:its}
For the tangent prolongation of a double vector bundle $(D;A,B;M)$ with core $C$, as shown in 
Figure~\ref{fig:tvbexs}(a), two duals are shown in Figure~\ref{fig:dtp}. The first is the cotangent 
double of $D$ and is described in detail in \cite{Mackenzie:2005dts}. In (b), $T_A^\sol D\to TA$ 
denotes the dual of the tangent prolongation bundle $TD\to TA$ and likewise $T^\sol C\to TM$ is the 
dual of $TC\to TM$. 

The duality between $C$ and $C^*$ induces a duality between $TC\to TM$ and $T(C^*)\to TM$ and
induces an isomorphism $T(C^*)\to T^\sol C$ \cite{MackenzieX:1994,Mackenzie:GT}. Similarly there 
is an isomorphism $T(D\duer A)\to T_A^\sol D$. These combine into an isomorphism from the triple
vector bundle in (b) to the tangent prolongation of $(D\duer A;A,C^*;M)$. 

In the case where $D$ is the double tangent bundle $T^2M$, the $X, Y$ and $Z$ duals of $T^3M$
may be canonically identified with $T^*(T^2M)$, $T(T^*TM)$ and $T^2(T^*M)$. There are also
canonical isomorphisms from the first of these to the second and third, and to the three
triple vector bundles $T^*(T^*TM)$, $T^*(TT^*M)$ and $T(T^*T^*M)$. These isomorphisms are of a
different character to those treated in this paper and will be studied elsewhere. 

\begin{figure}[h]
\begin{center}
\subfloat[]
{\smallcube{T^*D}{D\duer B}{D\duer A}{C^*}{D}{B}{A}{M}}
\qquad
\subfloat[]
{\smallcube{T_A^\sol D}{T^\sol C}{TA}{TM}{D\duer A}{C^*}{A}{M}}
\caption{See Example \ref{ej:its}. \label{fig:dtp}}
\end{center}
\end{figure}

\end{ej}


\section{Decompositions of triple vector bundles} 
\label{sect:dtvb}

We know from \S\ref{sect:intro} that the first six elements listed in 
\eqref{eq:k4gens} are equal to the identity. 
In addition, the group generated by the last three elements is normal in $\dg_3$.  This 
can be proven by a direct calculation: using the fact that the first six elements are equal 
to the identity, it follows that: 
\begin{equation}  \label{eq:actiononK4}
\begin{split}
X(XYXZ)^{2}X^{-1} = (ZXZY)^{-2}, \\ 
Y(XYXZ)^2Y^{-1}=(YZYX)^{-2},  \\
 Z(XYXZ)^2Z^{-1}=(XYXZ)^{-2}. 
\end{split}
\end{equation}
Hence, $K_4$ is the subgroup generated by the last three elements in \eqref{eq:k4gens}.  
To calculate it we use a triple analogue 
of \ref{thm:main}.  

We start by determining the statomorphisms from a decomposed triple vector bundle to
itself. Denote the set of decompositions of the triple vector bundle $E$ by $\Dec(E)$.  

\begin{prop}
\label{prop:statos}
Each statomorphism $\phi$ from $\Bar{\BE}$ to itself is of the form  
\begin{multline}  \label{eq:e123}
\phi(e_1,e_2,e_3,e_{12},e_{13},e_{23},e_{123})
= (e_1,\,e_2,\,e_3,\\
e_{12}+\gamma(e_1,\, e_2),\,e_{13}+\beta(e_1,\, e_3),\,
e_{23}+\alpha(e_2,\, e_3),\\
e_{123}+\nu(e_3,\, e_{12})+\lambda(e_1,\, e_{23})+\mu(e_2,\, e_{13})+\rho(e_1,\, e_2,\, e_3))
\end{multline}
where 
\begin{gather*}
\gamma\co E_1\otimes E_2\to E_{12},\quad 
\beta\co E_1\otimes E_3\to E_{13}, \quad
\alpha\co E_2\otimes E_3\to E_{23}, \\
\lambda\co E_1\otimes E_{23}\to E_{123},\quad
\mu\co E_2\otimes E_{13}\to E_{123},\quad
\nu\co E_3\otimes E_{12}\to E_{123},
\end{gather*}
and $\rho\co E_1\otimes E_2\otimes E_3\to E_{123}$ are linear maps. 
\end{prop}

Let $G_3$ denote the set of all such $(\gamma,\beta,\alpha,\lambda,\mu,\nu,\rho)$. 
This is a group under the composition
\begin{equation}
(\gamma',\beta',\alpha',\lambda',\mu',\nu',\rho')\,(\gamma,\beta,\alpha,\lambda,\mu,\nu,\rho)
= (\gamma'',\beta'',\alpha'',\lambda'',\mu'',\nu'',\rho'')
\end{equation}
where $\gamma'' = \gamma'+\gamma,\,\dots,\,\nu'' = \nu'+\nu$ and 
\begin{multline*}
\rho''(e_1,e_2,e_3) = 
\rho(e_1,\, e_2,\, e_3) + \rho'(e_1,\, e_2,\, e_3)\\
+\lambda'(e_1,\,\alpha(e_2,\, e_3)) +
\mu'(e_2,\,\beta(e_1,\, e_3)) +
\nu'(e_3,\,\gamma(e_1,\, e_2)). 
\end{multline*}

Given two decompositions $P$ and $P_0$ of a triple vector bundle $E$, there is a
unique statomorphism $\phi\in G_3$ such that $P = \phi\circ P_0$. 
Thus $\Dec(E)$ is a simply transitive $G_3$--space, or torsor. 

As in the double case, we define $\C$ to be the category whose objects are triple 
vector bundles, and whose morphisms are statomorphisms. We denote by $\C^{op}$ or by
$\C^{-1}$ its opposite category. The dualization operations $X$, $Y$ and $Z$ can now 
be regarded as functors $\mscr{C} \to \mscr{C}^{op}$ or as functors 
$\mscr{C}^{op} \to \mscr{C}$.

Let $W \in \dg_3$. There is a unique permutation $\sigma \in S_4$ such that, if $E$ is 
a triple vector bundle with building bundles $\BE$, then $E^W$ is a triple vector bundle 
with building bundles
$$
E_{\sigma(1)},\,
E_{\sigma(2)},\,
E_{\sigma(3)},\,
E_{\sigma(1)\sigma(2)},\,
E_{\sigma(1)\sigma(3)},\,
E_{\sigma(2)\sigma(3)},\,
E_{\sigma(1)\sigma(2)\sigma(3)}.
$$
Define $\pi (W) := \sigma$ and define $\varepsilon_W = \pm 1$ to be the signature of $\pi(W)$.
We can now look at $W$ as a functor $\mscr{C} \to \mscr{C}^{\varepsilon_W}$.
We now give the triple analogue of Theorem \ref{thm:submain}. 

We know that $\Dec(E)$ is a torsor modelled on $G_3$, which is (as a set) the
sum of the seven $\cinf{M}$-modules
\begin{gather*}
\Ga(E_1^*\otimes E_2^*\otimes E_{12}),\quad 
\Ga(E_1^*\otimes E_3^*\otimes E_{13}), \quad
\Ga(E_2^*\otimes E_3^*\otimes E_{23}),\\ 
\Ga(E_1^*\otimes E_{23}^*\otimes E_{123}),\quad
\Ga(E_2^*\otimes E_{13}^*\otimes E_{123}),\quad
\Ga(E_3^*\otimes E_{12}^*\otimes E_{123}),\\
\Ga(E_1^*\otimes E_2^*\otimes E_3^*\otimes E_{123}). 
\end{gather*}
In the same way it is easy to establish that $\Dec(E^X)$ is a torsor modelled on a group
which is the sum of the modules
\begin{gather*}
\Ga(E_0^*\otimes E_2^*\otimes E_{02}),\quad 
\Ga(E_0^*\otimes E_3^*\otimes E_{03}), \quad
\Ga(E_2^*\otimes E_3^*\otimes E_{23}),\\ 
\Ga(E_0^*\otimes E_{23}^*\otimes E_{023}),\quad
\Ga(E_2^*\otimes E_{03}^*\otimes E_{023}),\quad
\Ga(E_3^*\otimes E_{02}^*\otimes E_{023}),\\
\Ga(E_0^*\otimes E_2^*\otimes E_3^*\otimes E_{023}),
\end{gather*}
(see Table \ref{table:tvbXYZ}). Up to rearrangements of 
the tensor products, this is a permutation of the first list: for example, 
$E_0^*\otimes E_2^*\otimes E_{02}$, the first space in the second list, is 
$E_{123}\otimes E_2^*\otimes E_{13}^*$, a rearrangement of 
the fifth space in the first list. 
A similar result holds for $E^Y$ and $E^Z$ and thus for $E^W$ for any word $W\in\dg_3$. 

In this way one proves the following result. 

\begin{prop} 
\label{thm:submain3-i} 
Let $W \in \dg_3$ and let $E$ be a triple vector bundle 
with building bundles $\BE$.
Then there is a canonical representation of the space of decompositions 
$\Dec(E^W)$ as a torsor modelled on $G_3$.  
\end{prop}


We now calculate $\theta_X \co G_3 \to G_3$.  
First, we explain some notation.  For any  $I \subseteq \{0,1,2,3\}$, we write $e_I$ for a generic element 
of $E_I$.  Given a linear map $\gamma \co E_1 \otimes E_2 \to E_{12}$, we can also think of it, for instance, 
as a map $\gamma \co E_{30} \otimes E_2 \to E_{023}$; we can also think of it as simply 
$\gamma \in \Gamma(E_1^* \otimes E_2^* \otimes E_{03}^*)$.   With this abuse of notation, we can write:
$$
\pair{e_{03}}{\gamma(e_1, e_2)}= \pair{e_1}{\gamma(e_2, e_{03})} = \gamma(e_1, e_2, e_{03}).
$$
If, in addition, $\nu \in \Gamma(E_{0}^* \otimes E_{3}^* \otimes E_{12}^*)$, there is a linear map 
$\gamma \nu \in \Gamma(E_0^* \otimes E_1^* \otimes E_2^* \otimes E_3^*)$ defined by:
$$
\gamma \nu (e_0, e_1, e_2, e_3) := \pair{\nu(e_0, e_3)}{\gamma(e_1, e_2)} = 
\pair{e_3 }{  \nu \left( e_0, \gamma(e_1, e_2) \right) } = \ldots 
$$

Let $g = (\gamma,\beta,\alpha,\lambda,\mu,\nu,\rho)$ be an element of $G_3$ and let 
$\phi = \phi_g \co \Bar{\BE} \to \Bar{\BE}$ be the corresponding statomorphism.  
Then $\theta_X(\phi) \co \Bar{\BE}^X \to \Bar{\BE}^X$ is another statomorphism.  
We can write $\theta_X(\phi) = \phi_{\tilde{g}}$  for some element 
$\Tilde{g} = (\Tilde{\mu}, \Tilde{\nu}, \Tilde{\alpha}, \Tilde{\lambda}, \Tilde{\gamma}, \Tilde{\beta}, \Tilde{\rho})$ 
in $G_3$.  In order to describe $\theta_X$, we want to write $\Tilde{g}$ in terms of $g$.

As we explained with double vector bundles, if $d \in \Bar{\BE}$ and $\delta \in \Bar{\BE}^X = \Bar{\BE} \duer E_{2,3}$ 
are two elements that project to the same element in $E_{2,3}$ (so that they can be paired over $E_{2,3}$), then we 
shall have $\pair{\delta}{d} = \pair{\theta_{X}(\phi)(\delta)}{\phi(d)}$.  Let us write 
$$
d=(e_1, e_2, e_3, e_{12}, e_{13}, e_{23}, e_{123}),\qquad 
\delta=(e_0, e_2, e_3, e_{02}, e_{03}, e_{23}, e_{023}).
$$  
Then $\phi(d)$ is given by \eqref{eq:e123}, whereas
\begin{multline*}
\theta_X(\phi) (\delta) 
= (e_0,\,e_2,\,e_3,
e_{02}+\Tilde{\mu}(e_0,\, e_2),\,e_{03}+\Tilde{\nu}(e_0,\, e_3),\,
e_{23}+\Tilde{\alpha}(e_2,\, e_3),\\
e_{023}+\Tilde{\beta}(e_3,\, e_{02})+\Tilde{\lambda}(e_0,\, e_{23})+\Tilde{\gamma}(e_2,\, e_{03})+\Tilde{\rho}(e_0,\, e_2,\, e_3))
\end{multline*}
First we notice that, in order to be able to pair $\theta_{X}(\phi)(\delta)$
and ${\phi(d)}$ over $E_{2,3}$, we need to have $\Tilde{\alpha} = \alpha$.
Now:
$$
\pair{\delta}{d} = 
\pair{e_{023}}{e_1} + \pair{e_{03}}{e_{12}} + \pair{e_{02}}{e_{13}} + \pair{e_0}{e_{123}},
$$
and $\pair{\theta_{X}(\phi)(\delta)}{\phi(d)}$ is equal to 
\begin{multline*}
\pair{e_{023}}{e_1} + \Tilde{\lambda}(e_0,  e_{23}, e_1) + \Tilde{\gamma}(e_{03},,e_2, e_1) + 
\Tilde{\beta}(e_{02}, e_3, e_1)  + \Tilde{\rho}(e_0, e_2, e_3, e_1)\\ 
+ \pair{e_{03}}{e_{12}} + \Tilde{\nu}(e_{03}, e_2, e_1) + \gamma(e_0, e_3, e_{12}) 
+ \gamma \Tilde{\nu} (e_0, e_2, e_3, e_1) + \pair{e_{02}}{e_{13}}\\ 
+ \Tilde{\mu}(e_{02}, e_3, e_1) + \beta(e_0, e_2, e_{13}) + \beta \Tilde{\mu}(e_0, e_2, e_3, e_1) + \pair{e_0}{e_{123}}\\
+ \lambda(e_0, e_{23}, e_1) + \mu(e_0, e_2, e_{13}) + \nu(e_0, e_3, e_{12}) + \rho(e_0, e_2, e_3, e_1).
\end{multline*}
For these two expressions to be equal we must have:
\begin{equation*}
\begin{split}
& \lambda + \Tilde{\lambda} = 0, \quad \quad \mu + \Tilde{\mu} = 0, \quad \quad \nu + \Tilde{\nu} = 0, \\
& \beta + \Tilde{\beta} = 0 , \quad \quad \gamma + \Tilde{\gamma} = 0, \quad \quad \rho + \Tilde{\rho} + \gamma \Tilde{\nu} + \beta \Tilde{\mu} = 0.
\end{split}
\end{equation*}

To conclude we solve these equations and we obtain the action of $\theta_X$ on $G_3$, which is summarized in the 
first row of Table \ref{table:XYZ}. 
The actions of $Y$ and $Z$ are obtained in the same way. 
Recall that the notation $\gamma\nu + \beta\mu - \rho$ indicates the map $\Tilde{\rho}$ given by 
$$
\Tilde{\rho}(v_0,e_2,e_3) = \gamma(\nu(v_0,e_3),e_2) + \beta(\mu(v_0,e_2),e_3) - \rho(v_0,e_2,e_3). 
$$


\begin{table}[h]
\begin{center}
\begin{tabular}{|R||R|R|R|R|R|R|Z||R||} 
\hline
g & \gamma & \beta &\alpha &\lambda & \mu & \nu &\rho  & 1,2,3\\
\hline
\theta_X(g) & -\mu & -\nu & \alpha & -\lambda & -\gamma & -\beta & \gamma\nu + \beta\mu - \rho & 0,2,3\\
\hline
\theta_Y(g) & -\lambda & \beta & -\nu & -\gamma & -\mu & -\alpha & \alpha\lambda +\gamma\nu - \rho & 1,0,3\\
\hline
\theta_Z(g) & \gamma & -\lambda & -\mu & -\beta & -\alpha & -\nu & \alpha\lambda + \beta\mu - \rho & 1,2,0\\
\hline
\end{tabular}
\end{center}
\caption{The action of $X,Y,Z$ on $G_3$. In the final column, $i,j,k$ indicates that the domain of $\rho$ 
is $E_i\otimes E_j\otimes E_k$; this is needed in the proof of Proposition \ref{prop:xyxz}. \label{table:XYZ}}
\end{table}

The next two results can now be proved in the same way as in \S\ref{sect:intro}. 

\begin{thm} \label{thm:submain3} 
Let $W \in \dg_3$ and let $E$ be a triple vector bundle in $\mscr{C}$.  
\begin{enumerate}
\item  The bijection $\vartheta_W \co \Dec(E) \to \Dec(E^W)$ defined by 
$$
\vartheta_W(P) = (P^W)^{\varepsilon_W}, 
$$
is a map of torsors and the associated group automorphism $\theta_W \co G_3\to G_3$ does not depend on the 
choice of triple vector bundle $E$. 

\item  The map $\dg_3 \times G_3\to G_3,\quad (W,g) \mapsto \theta_W(g),$
is a group action. 
\end{enumerate}
\end{thm}

\begin{thm} \label{thm:main3}
For $W \in \dg_3$ with $\pi(W) = 1$, the following are equivalent:
\begin{enumerate}
\item $\theta_W$ is the identity on $G_3$. 
\item There exists a natural isomorphism between the identity functor and $W$.
\end{enumerate}
\end{thm}

Using Table~\ref{table:XYZ} we can determine which words act as the identity. 

\begin{prop}
\label{prop:xyxz}
The words $(XYXZ)^2,\ (YZYX)^2,\ (ZXZY)^2$ have order $2$ in $\dg_3$, and 
\begin{equation}
\label{eq:K4}
(YZYX)^2 (XYXZ)^2= (ZXZY)^2.
\end{equation}
\end{prop}

\pf
Consider $(XYXZ)^2$. From Table~\ref{table:XYZ} the first six columns of Table~\ref{table:xyxz2} 
are easily established. 

\begin{table}[H]
\begin{center}
\begin{tabular}{|R||R|R|R|R|R|R|Z||} 
\hline
 & \gamma & \beta &\alpha &\lambda & \mu & \nu &\rho  \\
\hline
X & -\mu & -\nu & \alpha & -\lambda & -\gamma & -\beta & \gamma\nu + \beta\mu - \rho\\
\hline
YX & \mu & \alpha & -\nu & \gamma & \lambda & -\beta & \rho - \beta\mu -\gamma\nu\\ 
\hline
XYX & -\gamma & \alpha & \beta & -\mu & -\lambda & \nu & -\rho\\
\hline
XYXZ & -\gamma & \mu & \lambda & -\alpha & -\beta & -\nu & \rho - \alpha\lambda - \beta\mu\\
\hline
(XYXZ)^2 & \gamma & -\beta & -\alpha & -\lambda & -\mu & \nu & \rho\\
\hline
\end{tabular}
\end{center}
\caption{Calculation of the action of $(XYXZ)^2$. We omit $\theta$ from the notation.\label{table:xyxz2}}
\end{table}

Now consider the final column. The operation of $X$ sends $\rho$ to $\gamma\nu + \beta\mu - \rho$,
considered as a map $E_0\otimes E_2\otimes E_3\to E_{023}$. Applying $Y$ to this we get
\begin{equation}
\label{eq:rhoXY}
(-\lambda)(-\alpha) + \beta(-\mu) - (\alpha\lambda + \gamma\nu - \rho),
\end{equation}
which should be considered as a map $E_0\otimes E_1\otimes E_3\to E_{013}.$

The term $\lambda\alpha$ in (\ref{eq:rhoXY}), when considered as a map $E_0\otimes E_1\otimes E_3\to E_{013}$, 
is
$$
\langle (\lambda\alpha)(e_0,e_1,e_3)\bk e_2\rangle = \langle \lambda(e_1,\alpha(e_2,e_3))\bk e_0\rangle.
$$
On the other hand, the term $\alpha\lambda$ in (\ref{eq:rhoXY}), when considered as a map $E_0\otimes E_1\otimes E_3\to E_{013}$, 
is 
$$
\langle (\alpha\lambda)(e_0,e_1,e_3)\bk e_2\rangle = \langle \alpha(\lambda(e_0,e_1),e_3))\bk e_2\rangle.
$$
These are equal and so we may briefly write $\alpha\lambda = \lambda\alpha$. Thus (\ref{eq:rhoXY})
simplifies to $\rho - \beta\mu -\gamma\nu.$

The remaining three entries in the $\rho$ column are obtained in the same way, and it is now 
clear that $(XYXZ)^2$ has order 2.

For reference, we list the actions of all three elements in Table~\ref{table:K4}. 

\begin{table}[h]
\begin{center}
\begin{tabular}{|R||R|R|R|R|R|R|Z||} 
\hline
         & \gamma & \beta &\alpha &\lambda & \mu & \nu &\rho  \\
\hline
(XYXZ)^2 & \gamma & -\beta & -\alpha & -\lambda & -\mu & \nu & \rho\\
\hline
(YZYX)^2 & -\gamma & -\beta & \alpha & \lambda & -\mu & -\nu & \rho\\
\hline
(ZXZY)^2 & -\gamma & \beta & -\alpha & -\lambda & \mu & -\nu & \rho\\
\hline
\end{tabular}
\end{center}
\caption{Actions on $G_3$ of the nonidentity elements of $K_4$. \label{table:K4}}
\end{table}

We prove (\ref{eq:K4}) by applying $XYX = YXY$ and its conjugates repeatedly to the LHS. 
Eventually we arrive at $(YZXZ)^2$. Since this has order 2 and each of $X,Y,Z$ have order 2, 
this is equal to $(ZXZY)^2$. 
\pfend

We remark that, given the action of $X$ and the action of $Y$ on $G_3$, there are two ways 
to understand how the action of the composition $YX$ should be calculated.  They are duals 
of each other, and they correspond to thinking in terms of ``frames'' or thinking in 
terms of ``coordinates''.  If the reader is obtaining different calculations, this may be 
the reason.  They are equivalent, however, and they both should agree on $K_4$.  

From $(XYXZ)^2\neq I$ it follows that $XYX$ and $Z$ do not commute. Notice however that
the actions of $XYX$ and $Z$ on the set of building bundles do commute. The action of 
$XYX$ on the decomposed triple vector bundle, as distinguished from the action on the statomorphisms, 
takes place entirely in the roof double vector bundle (in terms of Figure~\ref{fig:tvb}) and 
preserves the floor double vector bundle. The action on $G_3$ however shows that the operations 
of $XYX$ and $Z$ are `entangled'. 

From \ref{prop:xyxz} it follows that $K_4$ is the Klein 4-group. As an immediate
consequence we have:

\begin{thm}
The order of $\dg_3$ is $96.$
\end{thm}

In \cite{Mackenzie:2005dts} the order of $\dg_3$ (there denoted $\EuScript{VB}_3$) was given 
as 72. Further, it was stated that $(XYZ)^4 = I$. Applying Table~\ref{table:XYZ}, the action 
of $(XYZ)^4$ on $G_3$ is as shown in Table~\ref{table:xyz4}, so $(XYZ)^4$ has order 2.   
More specifically, $(XYZ)^4=(ZXZY)^2$ is a non-trivial element of $K_4$.

\begin{table}[h]
\begin{center}
\begin{tabular}{|R||R|R|R|R|R|R|Z||} 
\hline
         & \gamma & \beta &\alpha &\lambda & \mu & \nu &\rho  \\
\hline
(XYZ)^4 & -\gamma & \beta & -\alpha & -\lambda & \mu & -\nu & \rho \\
\hline
\end{tabular}
\end{center}
\caption{Action of $(XYZ)^4$ on $G_3$. \label{table:xyz4}}
\end{table}

\begin{thm} \  \label{thm:descriptionDG3}
\begin{enumerate}
\item The group $\dg_3$ is an  
extension of $S_4$ by $K_4$; that is, there is a short exact sequence 
$\ses{K_4}{\dg_3}{S_4}$.   
\item  As an $S_4$--module, $K_4$ is isomorphic to the normal subgroup of $S_4$ 
$$\{  1, (12)(30), (23)(10), (13)(20) \}$$ with action by conjugation.
\item  The extension $\ses{K_4}{\dg_3}{S_4}$ is not split.
\end{enumerate}
\end{thm} 

\pf
(i) has already been established, and (ii) follows from \eqref{eq:actiononK4}. 
We now prove (iii).

Assume the extension is split.  Let $f \co \dg_3 \to K_4 \rtimes S_4$ be an isomorphism, 
where $\rtimes$ represents semidirect product.   Let $\sigma=(12)(30)$.
Then $f(XYXZ) = (u, \sigma)$ for some element $u \in K_4$.  Using the semidirect product rule:
$$f((XYXZ)^2)=\left(u,\sigma\right)^2=(u (\sigma \cdot u), \sigma^2) = (u^2, 1) = (1, 1),$$
where we have used the fact that $\sigma$ acts trivially on $K_4$.  We have reached 
a contradiction.  
\pfend

The three nontrival elements of $K_4$ are in some ways comparable to the dualization of 
an ordinary vector bundle $A$: isomorphisms between $A$ and $A^*$ exist but are not natural, 
and likewise there are statomorphisms between a triple vector bundle $E$ and $E^W$ where 
$W\in K_4,\, W\neq 1,$ but there is no canonical statomorphism. However, whereas ordinary 
dualization is contravariant, the elements of $K_4$ are covariant functors. 

For arbitrary $n$--fold vector bundles, (i) of Theorem \ref{thm:descriptionDG3} will still 
be true \cite{Gracia-SazM:nfold}. Instead of (ii) we will give a combinatorial description 
of $K_{n+1}$.  The corresponding extension will be split if and only if $n$ is even.


\section{The group $\dg_3$}
\label{sect:group}

In this section we identify $\dg_3$ as a semi-direct product and give its normal and
conjugacy class structure. 

First, consider the words $a=(ZXY)^2$, $b=(XYZ)^2$, and $c=(YZX)^2$.  Let $H$ be the 
subgroup of $\dg_3$ generated by $a$, $b$, and $c$.  These three elements have order 
4, commute with each other, and satisfy $abc=1$.   Thus $H \cong \zz_4 \times \zz_4$.   
A direct calculation shows that $H$ is a normal subgroup of $\dg_3$.   Second, 
consider $\dg_2 \cong S_3$ as a subgroup of $\dg_3$ generated by $X$ and $Y$.  
Then $S_3$ acts on $H$ by permuting $a$, $b$, and $c$.   
In addition, $H \cap \dg_2 = \{ 1 \}$, producing:

\begin{thm}
$\dg_3$ is isomorphic to the semidirect product $(\mathbb{Z}_4 \times \mathbb{Z}_4) \rtimes S_3$, 
as described above.
\end{thm}

All normal subgroups of $\dg_3$ can be obtained with help of the homomorphism 
$\pi \co \dg_3 \to S_4$.

\begin{thm}
The normal subgroups of $\dg_3$, apart from the trivial ones and $K_4 = \ker \pi$, are
\begin{itemize}
\item $\pi^{-1} (A_4)$, which has index 2.  This consists of the words in $X$, $Y$, 
and $Z$ that have an even number of letters.  Alternatively, this consists of the 
duality functors which are covariant.  
\item The subgroup $H$ described above, which has index 6, and which is also 
$\pi^{-1}(V)$, where $V$ is the normal subgroup of $S_4$ of order 4.
\end{itemize}
\end{thm}
\pf
It is clear that these are normal subgroups.  We obtained that there are 
no others by use of  GAP \cite{GAP449}.
\pfend

Also from GAP, or equally by hand calculation, 
we find that there are nine nontrivial 
conjugacy classes. In Table \ref{table:cc} we have listed their size, representative 
elements, their orders, and the action of these elements on $G_3$.   
We can read all the normal subgroups from this table.  Apart from the identity, $K_4$ consists 
of the first conjugacy class, $H$ consists of the first four conjugacy classes, 
and $\pi^{-1}(A_4)$ consists of the first five conjugacy classes.

\begin{table}[h]
\begin{center}
\begin{tabular}{||Z|R|R|R|R|R|R|R|R|Z||} 
\hline
& \textrm{s}  & \textrm{o} & \alpha & \beta &\gamma &\lambda & \mu & \nu &\rho  \\
\hline\hline
  (XYXZ)^2 & 3 & 2 & -\alpha & -\beta & \gamma & -\lambda & -\mu & \nu & \rho \\
\hline
 (XYZ)^2 & 3 & 4 & -\lambda & \beta & \nu & \alpha & \mu & -\gamma & \rho - \alpha \lambda - \gamma \nu \\
\hline
 (ZYX)^2 & 3 & 4 & \lambda & \beta & -\nu & -\alpha & \mu & \gamma & \rho - \alpha \lambda - \gamma \nu \\
\hline
 XZXY & 6 & 4 &  \lambda & -\mu & \nu & -\alpha & -\beta & -\gamma & \rho - \alpha \lambda - \gamma \nu \\ 
\hline
 XY & 32 & 3 & \beta & -\nu & \lambda & \mu & \gamma & -\alpha & \rho - \alpha \lambda - \gamma \nu \\
\hline
 Z & 12 & 2 &  -\mu & -\lambda & \gamma & -\beta & -\alpha & -\nu & -\rho + \alpha \lambda + \beta \mu \\
\hline
 XYZYXZY & 12 & 4 & \mu & -\lambda & -\gamma & \beta & -\alpha & \nu & -\rho + \alpha \lambda + \beta \mu \\  
\hline
 XYZ & 12 & 8 & -\gamma & -\mu & -\lambda & \nu & -\beta & \alpha & -\rho + \alpha \lambda + \beta \mu \\
\hline
 ZYX & 12 & 8 & \nu & -\mu & -\alpha &  \gamma & -\beta & \lambda & -\rho + \beta \mu + \gamma \nu \\
\hline
\end{tabular}
\end{center}
\caption{Conjugacy class structure of $\dg_3$. The leftmost column gives a representative
element of each class, followed by the size of the class (s), the order of any element of 
the class (o), and the action of the representative on $G_3$.\label{table:cc}}   
\end{table}

To conclude, there is a faithful linear representation of $\dg_3$ on $\rr^6$.  
Neglecting the $\rho$ column of Table~\ref{table:XYZ}, the action of $X,Y,Z\in \dg_3$ on $G_3$ defines
a linear action on $\rr^6$ by matrices with entries $0, +1$ or $-1$. 
From Table~\ref{table:cc}, the only element of $\dg_3$ to fix all of
$\gamma,\dots,\nu$ is the identity. 



\newcommand{\noopsort}[1]{} \newcommand{\singleletter}[1]{#1} \def\cprime{$'$}
  \def\cprime{$'$}
\providecommand{\bysame}{\leavevmode\hbox to3em{\hrulefill}\thinspace}
\providecommand{\MR}{\relax\ifhmode\unskip\space\fi MR }
\providecommand{\MRhref}[2]{%
  \href{http://www.ams.org/mathscinet-getitem?mr=#1}{#2}
}
\providecommand{\href}[2]{#2}

%

\end{document}